\documentclass[11pt]{amsart}

\usepackage{amsmath}
\usepackage{amssymb}
\usepackage{graphicx}

\hoffset = - 0.5in

\newtheorem{theorem}{Theorem}[section]
\newtheorem{corollary}[theorem]{Corollary}

\newtheorem{proposition}[theorem]{Proposition}
\theoremstyle{definition}
\newtheorem{definition}[theorem]{Definition}

\numberwithin{equation}{section}

\renewcommand{\geq}{\geqslant}
\renewcommand{\leq}{\leqslant}

\title{Properties of $R(X), R(a, X) \mbox{ and  } RW(a, X)$}

\author[T. Dalby]{Tim Dalby}

\date{\today}



\email{tim\_dalby@bigpond.com}

\begin{document}

\parindent = 0pt
\parskip = 8pt

\begin{abstract}

In the context of metric fixed point theory in Banach spaces three moduli have played an important role.  These are R(X), R(a, X) and RW(a, X). This paper looks at some of their properties.  Also investigated is what happens when they take on the value of 1.  The situation where these moduli are set in dual space is also considered.

\end{abstract}

\maketitle

\section{Introduction}   
   
The Banach space moduli of  $R(X), R(a, X) \mbox{ and  } RW(a, X)$ were introduced within the context of fixed point theory.  This area of research looks at whether, in an infinite dimensional real Banach space, every nonexpansive mappings on every weak compact convex nonempty set has a fixed point.  If this is so then the space is said to have the weak fixed point property (w-FPP).  It can be shown that this propery is separably determined so in this paper the Banach space is assumed to be separable.  

If the sets are not necessarily weak compact but closed then the property is called the fixed point property (FPP).

Another aspect of this topic is that it involves weakly convergent sequences and usually a translation is employed so that the sequence being studied is weakly convergent to zero.  Inherent in most results is the interplay between weak null sequences and the norm.  Much more background can be found in Goebel and Kirk [14] and  Kirk and Sims [18].
   
These moduli all measure a relationship between weak null sequences and the norm and they have similar definitions.  These definitions are given in the next section.  Because of the close similarity of the definitions it is natural to ask how the three $Rs$ are related, do they share similar properties and, if they take on certain values, what does mean for that Banach space.  This paper explores these questions.

\section{Definitions}
 
\begin{definition}
   
[Garc\'{i}a-Falset, 10]
\[ R(X) = \sup\{\liminf_{n \rightarrow \infty}\| x_n + x \| : \| x \| \leq 1, \| x _n \| \leq 1 \mbox{ for all } n, x_n \rightharpoonup 0\}. \]
\end{definition}

So $1 \leq R(X) \leq 2$.

In [10] Garc\'{i}a-Falset showed that if a Banach space, $X$, has the nonstrict Opial property and $R(X) < 2$ then $X$ has the w-FPP.  About the same time Prus [24] produced a similar result.  Later Garc\'{i}a-Falset [11] was able to drop the nonstrict Opial condition.

Dom\'{i}nguez-Benavides in [7] created a new version of $R(X)$ this time with a parameter and an extra condition on the weak null sequences.

\begin {definition}

[Dom\'{i}nguez-Benavides, 7] 
\[ \mbox{ For } a \geq 0 \mbox{ let } R(a, X) = \sup\{\liminf_{n \rightarrow \infty}\| x_n + x \| : \| x \| \leq a, \| x _n \| \leq 1 \mbox{ for all } n, x_n \rightharpoonup 0 \]
\[ \mbox{ where } D[(x_n)] = \limsup_{n \rightarrow \infty} \left ( \limsup_{m \rightarrow \infty}\| x_n - x_m \| \right ) \leq 1 \}. \]

\end{definition}

In [7] Dom\'{i}nguez-Benavides showed that if $R(a, X) < 1 + a$ for some $a \geq 0$ then $X$ has the w-FPP.  In addition, it was shown that there is a link between $R(a, X)$ and the modulus of near uniform smoothness, $\Gamma(t)$, in reflexive Banach spaces.  This link was explored further in [12] where the third one of the $Rs$ was defined.

\begin {definition}

[Garc\'{i}a-Falset, Llorens-Fuster and Mazcu\~{n}an-Navarroa, 12]
\[ \mbox{ For }a > 0 \mbox{ let } RW(a, X) = \sup\{\liminf_{n \rightarrow \infty}\| x_n + x \| \wedge \liminf_{n \rightarrow \infty}\| x_n - x \| : \| x \| \leq a, \]
\[ \| x _n \| \leq 1 \mbox{ for all } n, x_n \rightharpoonup 0 \}. \]

\end{definition}

In all three definitions $\liminf$ can be replaced by $\limsup.$

Again there is a link between $RW(a, X)$ and $\Gamma(t)$.  The authors exploited this link with $\Gamma(t)$ to show that uniformly nonsquare Banach spaces have the FPP and thus solved a long standing conjecture.

In [12] it was shown that $R(a, X) \leq RW(a, X) \mbox{ for all } a > 0.$  Using this inequality and the definitions it is clear that $R(1, X) \leq RW(1, X) \leq R(X).$

So $RW(a, X)$ is the newest modulus and sits in the ``middle'' between $R(a, X)$ and $R(X).$  It is probably the most useful of the three because of its links to $\Gamma(t)$ and $\Gamma'(t).$  This is discussed in Section 4.

The following results are motivated by corollary 4.3 of [12].

\section{Properties of $R(a, X)$}

It is straightforward to see that
\[ 1 \vee a \leq R(a, X) \leq 1 + a \]
and
\[ \mbox{ if } 0 \leq a \leq b \mbox{ then } R(a, X) \leq R(b, X). \]

\begin{proposition}

Let $X$ be a separable Banach space then there exist $a > 0$ such that $R(a, X) = 1 + a$ if and only if $R(b, X) = 1 + b$ for all $b > 0.$

\end{proposition}

\begin{proof}

This proof is based on part of the proof of corollary 4.3 of [12].

Assume there exist $a > 0$ such that $R(a, X) = 1 + a.$  Let $\eta \in (0, 1)$ then there exists 
\[ x \in X , \| x \| \leq a, x_n \rightharpoonup 0, \| x_n \| \leq 1 \mbox{ for all } n \mbox { and } \limsup_{n \rightarrow \infty} \left ( \limsup_{m \rightarrow \infty}\| x_n - x_m \| \right ) \leq 1 \]
such that 
\[ \liminf_{n \rightarrow \infty}\| x_n + x \| > 1 + a - \eta (a \wedge 1). \qquad \mbox{ \# } \]  

Let $x_n^* \in X^*, \| x_n^* \| = 1, x_n^*(x_n + x) = \| x_n + x \| \mbox{ for all } n.$

Because $X$ is separable the dual unit ball, $B_{X^*}$, is weak* sequentially compact so without loss of generality $x_n^*$ can be assumed to converge weak* to $x^* \in B_{X^*}.$

Now $x_n^*(x_n)  \leq \| x_n \| \leq 1, x^*(x) \leq \| x \| \leq a.$

From \#
\begin{align*}
\liminf_{n \rightarrow \infty} x_n^*(x_n + x)  & > 1 + a - \eta (a \wedge 1) \\
\liminf_{n \rightarrow \infty} x_n^*(x_n)   +  x^*(x) & > 1 + a - \eta (a \wedge 1)  \qquad \mbox{ \#\# }\\
\liminf_{n \rightarrow \infty} x_n^*(x_n) & > 1 + a - \eta (a \wedge 1) - x^*(x)\\
 & \geq  1 + a - \eta (a \wedge 1) - a \\
 & = 1 - \eta (a \wedge 1) \\
 & \geq 1 - \eta.
\end{align*}

Next, from \#\#
\begin{align*}
 x^*(x) & > 1 + a - \eta (a \wedge 1) - \liminf_{n \rightarrow \infty}x_n^*(x_n) \\
 & \geq 1 + a - \eta (a \wedge 1) - 1 \\
& = a - \eta (a \wedge 1) \\
& \geq a - a \eta \\
& = a(1 - \eta).
\end{align*}

Now let $b > 0$ then

\begin{align*}
\liminf_{n \rightarrow \infty} \left \| x_n + \frac{b}{a} x \right \| & \geq \liminf_{n \rightarrow \infty} x_n^*(x_n + \frac{b}{a} x) \\
& = \liminf_{n \rightarrow \infty} x_n^*(x_n) + \frac{a}{b} x^*(x) \\
& > 1 - \eta + \frac{b}{a} \times a (1 - \eta) \\
& = (1 - \eta)(1 + b).
\end{align*}

Note that $\left \| \frac{b}{a} x \right \| \leq \frac{b}{a} \times a = b$ so

\[ R(b, X) \geq \liminf_{n \rightarrow \infty} \left \| x_n + \frac{b}{a} x \right \|  > (1 - \eta)(1 + b). \]

Allowing $\eta \rightarrow 0$ gives $ R(b, X) \geq 1 + b.$

But $R(b, X) \leq 1 + b \mbox{ for all } b  > 0.$  Therefore $R(b, X) = 1 + b.$

\end{proof}

\begin{corollary}

Let $X$ be a separable Banach space then there exist $a > 0$ such that $R(a, X) < 1 + a$ if and only if  $R(b, X) < 1 + b$ for all $b > 0.$
\end{corollary}

\begin{proof}

Only one of the implications needs to be considered.  Assume there exist $a > 0$ such that $R(a, X) < 1 + a$ and also assume there exists a $b > 0$ such that  $R(b, X) = 1 + b.$  Then by proposition 3.1 $R(a, X) = 1 + a.$  This is a contradiction.

\end{proof}

\begin{corollary}

Let $X$ be a separable Banach space then the following are equivalent.

\begin{enumerate}

\item[(a)] There exists $a > 0$ such that $R(a, X) < 1 + a$;

\item[(b)] $R(a, X) < 1 + a$ for all $a > 0$;

\item[(c)] $R(1, X) < 2.$

\end{enumerate}

\end{corollary}

{\bf Remarks}

\begin{enumerate}

\item[1.] The equivalence of (b) and (c) was mentioned by Prus in his talk at the 12th International Conference on Fixed Point Theory and Its Applications [25].

\item[2.] It would appear that the introduction of the parameter $a$ does not create a great advantage in fixed point theory.  The real advantage is using $D[(x_n)]$ in the definition of $R(a, X).$

\end{enumerate}

\begin{corollary}

Let $X$ be a separable Banach space then the following are equivalent.

\begin{enumerate}

\item[(a)] There exists $a > 0$ such that $R(a, X) = 1 + a$;

\item[(b)] $R(a, X) = 1 + a$ for all $a > 0$;

\item[(c)] $R(1, X) = 2.$

\end{enumerate}

\end{corollary}

\begin{corollary}

Let $X$ be a separable Banach space then either 
\[R(a, X) = 1 + a \mbox{ for all } a > 0 \mbox{ or } R(a, X) < 1 + a \mbox{ for all } a > 0.\]

\end{corollary}

\begin{corollary}

Let $X$ be a separable Banach space then either 
\[R(a, X) = 1 + a \mbox{ for all } a > 0 \mbox{ or } R(1, X) < 2.\]

\end{corollary}

So if $R(a, X)$ takes on its largest possible value then some interesting conclusions can be drawn.  At the other end of the scale, when $R(a, X)$ takes on one of its possible minimum values, $a$, there is also an interesting result.  The case when the $Rs$ equal 1 when $a = 1$ is explored further in a later section.

Another point of interest is the rate of increase of $R(a, X).$  The following proposition looks into that and then the consequences are investigated.

\begin{proposition}

Let $X$ be a separable Banach space then $0 < a \leq b$ implies $a R(b, X) \leq b R(a, X).$ 

\end{proposition}

\begin{proof}

Let $X$ be a separable Banach space and $a \mbox{ and } b$ be real numbers such that $0 < a \leq b.$

Consider
\[ x \in X , \| x \| \leq b, x_n \rightharpoonup 0, \| x_n \| \leq 1 \mbox{ for all } n \mbox { and } \limsup_{n \rightarrow \infty} \left ( \limsup_{m \rightarrow \infty}\| x_n - x_m \| \right ) \leq 1. \]

Then
\[ \left \| \frac{a}{b} x \right \| \leq \frac{ a }{ b } \times b = a,  \frac{ a }{ b } x_n \rightharpoonup 0, \left \| \frac{ a }{ b } x_n  \right \| \leq \frac{ a }{ b } \leq 1 \mbox{ for all } n \mbox { and }\]

\[ \limsup_{n \rightarrow \infty} \left ( \limsup_{m \rightarrow \infty} \left \| \frac{ a }{ b } x_n - \frac{ a }{ b } x_m \right \| \right ) = \frac{ a }{ b } \limsup_{n \rightarrow \infty} \left ( \limsup_{m \rightarrow \infty} \left \| x_n - x_m \right \|\right ) \leq \frac{ a }{ b } \leq 1. \]

Thus
\begin{align*}
\liminf_{n \rightarrow \infty} \| x_n + x \| & = \frac{ b }{ a } \liminf_{n \rightarrow \infty}\left \| \frac{ a }{ b } x_n + \frac{ a }{ b } x \right \| \\
& \leq \frac{ b }{ a } R(a, X).
\end{align*}

Therefore $R(b, X) \leq \frac{b}{a} R(a, X).$
 
\end{proof}

This means that if $0 < a \leq b$ then

\[ \frac{ R(b, X) }{ b } \leq \frac{ R(a, X) }{ a }. \]

In other words $\frac{ R(a, X) }{ a }$ is nonincreasing.  The consequences of this property deserves further investigation.

\begin{corollary}

Let $X$ be a separable Banach space then if there exists $a > 0$ such that $R(a, X) = a$ then $R(b, X) = b$ for all $b \geq a.$

\end{corollary}

\begin{proof}

 Assume there exists an $a > 0$ such that $R(a, X) = a.$  Let $b \geq a.$  Then by proposition 3.7 $R(b, X) \leq \frac{ b }{ a } R(a, X) = \frac{ b }{ a } \times a = b$ but $R(b, X) \geq b \mbox{ for all } b > 0.$

So $R(b, X) = b.$

\end{proof}

\begin{corollary}

Let $X$ be a separable Banach space then $R(a, X) = a \mbox{ for all } \newline a \geq 1 \mbox{ if and only if } R(1, X) = 1.$

\end{corollary}
 
\begin{proof}
 
Only one implication needs to be considered.  Let $a \geq 1$ then proposition 3.7 means $1 \times R(a, X) \leq a \times R(1, X) = a$ but $R(a, X) \geq a \mbox{ for all } a > 0.$

So $R(a, X) =a.$

\end{proof}

\section {Properties of $RW(a, X)$}

It is straightforward to see that
\[ 1 \vee a \leq RW(a, X) \leq 1 + a \]
and
\[ \mbox{ if } 0 \leq a \leq b \mbox{ then } RW(a, X) \leq RW(b, X). \]

$RW(a, X)$ is closely related to $\Gamma_X (t)$ and $\Gamma_X' (t)$ in reflexive Banach spaces.

The modulus of nearly uniform smoothness is

\[ \Gamma_X (t) = \sup \left\{ \inf_{n > 1} \left( \frac{ \| x_1 + t x_n \| + \| x_1 - t x_n \| }{ 2 } - 1 \right) \right\}, \]

where the supremum is taken over all basic sequences $(x_n)$ in $B_X.$

If $X$ is reflexive then

\[ \Gamma_X (t) = \sup \left\{ \inf_{n > 1} \left( \frac{ \| x_1 + t x_n \| + \| x_1 - t x_n \| }{ 2 } - 1 \right): (x_n) \mbox{ in } B_X, x_n \rightharpoonup 0 \right\}. \]

The definition of $\Gamma_X'(t)$ is

 \[ \Gamma_X' (0) = \lim_{t \rightarrow 0 ^+} \frac{ \Gamma_X(t) }{ t }. \]
 
In [21] and [12] it was shown that, in a reflexive Banach space, there exists $a > 0$ such that $RW(a, X) < 1 + a$ if and only $\Gamma_X '(0)  < 1$ if and only if there exists $t  > 0$ such that $\Gamma_X (t) < t.$

The first condition on $RW(a, X)$ will be looked at below.

Not surprisingly these properties are similar to those of $R(a, X)$ and the proofs are similar.  To prevent boredom a slightly different approach is taken.

\begin{proposition}

Let $X$ be a separable Banach space then the following are equivalent.

\begin{enumerate}

\item[(a)] $RW(a, X) < 1 + a$ for all $a > 0$;

\item[(b)] There exists $a > 0$ such that $RW(a, X) < 1 + a$;

\item[(c)] $RW(1, X) < 2.$

\end{enumerate}

\end{proposition}

\begin{proof}

Both (a) $\Rightarrow$ (b) and (a) $\Rightarrow$ (c) are straightforward.

$\bf{(b) \Rightarrow (a):}$ Assume that (a) does not hold then there exists $b > 0$ such that $RW(b, X) = 1 + b.$  Then $RW(c, X) = 1 + c \mbox{ for all } c > 0.$  This result formed part of the proof of corollary 4.3 of [12].  So there is a contradiction.

$\bf{(c) \Rightarrow (a):}$  This is proved in a similar fashion.

\end{proof}

Another point of interest is the rate of increase of $RW(a, X)$ and not surprisingly it is the same as for $R(a, X). $

\begin{proposition}

Let $X$ be a separable Banach space then $0 < a \leq b$ implies $a RW(b, X) \leq b RW(a, X).$ 

\end{proposition}

\begin{proof}

Consider $x \in X , \| x \| \leq b, x_n \rightharpoonup 0, \| x_n \| \leq 1 \mbox{ for all } n.$

Note that $\left \| \frac{a}{b} x \right \| \leq \frac{a}{b} \times b \leq a, \frac{a}{b} x_n \rightharpoonup 0, \| \frac{a}{b} x_n \| \leq \frac{a}{b} \leq 1 \mbox{ for all } n.$

\begin{align*}
\liminf_{n \rightarrow \infty}\| x_n + x \| \wedge \liminf_{n \rightarrow \infty}\| x_n - x \| & = \frac{ b }{ a } \left (\liminf_{n \rightarrow \infty}\| \frac{ a }{ b } x_n + \frac{ a }{ b } x \| \wedge \liminf_{n \rightarrow \infty}\| \frac{ a }{ b } x_n - \frac{ a }{ b } x \| \right) \\
& \leq \frac{ b }{ a } RW(a, X).
\end{align*}

Therefore $RW(b, X) \leq \frac{ b }{ a } RW(a, X).$

\end{proof}

This means that if $0 < a \leq b$ then

\[ \frac{ RW(b, X) }{ b } \leq \frac{ RW(a, X) }{ a }. \]

In other words $\frac{ RW(a, X) }{ a }$ is nonincreasing. 

There is a tantalising similarity to a property of $\Gamma_X (t).$  In [12] it was shown that $\frac{ \Gamma_X (t)}{ t}$ is nondecreasing.  These properties of $RW(a, X)$ and $\Gamma_X (t)$ combined with the two inequalities linking them in theorem 4.1 of [12] provide fertile ground for further research.

Now for some consequences of proposition 4.2.

\begin{corollary}

Let $X$ be a separable Banach space then if there exists $a > 0$ such that $RW(a, X) = a$ then $RW(b, X) = b$ for all $b \geq a.$

\end{corollary}

The proof is identical to the proof of corollary 3.8 and so it will be omitted.

\begin{corollary}

Let $X$ be a separable Banach space then $RW(a, X) = a \mbox{ for all } \newline a \geq 1 \mbox{ if and only if } RW(1, X) = 1.$

\end{corollary}

\begin{corollary}

Let $X$ be a separable Banach space then the following are equivalent.

\begin{enumerate}

\item[(a)] $RW(a, X) = 1 + a$ for all $a > 0$;

\item[(b)] There exists $a > 0$ such that $RW(a, X) = 1 + a$;

\item[(c)] $RW(1, X) = 2.$

\end{enumerate}

\end{corollary}

{\bf Remark}   Some of the last corollary is mentioned in the proof of corollary 4.3 of [12].

\begin{corollary}

Let $X$ be a separable Banach space then either 
\[RW(a, X) = 1 + a \mbox{ for all }a > 0 \mbox{ or } RW(a, X) < 1 + a \mbox{ for all } a > 0.\]

\end{corollary}

\begin{corollary}

Let $X$ be a separable Banach space then either 
\[RW(a, X) = 1 + a \mbox{ for all }a > 0 \mbox{ or } RW(1, X) < 2.\]

\end{corollary}

In [10] proposition III-7 states that $R(X) = 1$ and $X$ having a Kadec-Klee norm is equivalent to $X$ having the Schur property.   This result can now be rewritten in terms of $RW(a, X).$

\begin{proposition}

Let $X$ be a separable Banach space then the following are equivalent.

\begin{enumerate}

\item[(a)] There exists $a > 0$ such that $RW(a, X) = a$ and $X$ has Kadec-Klee norm;
\item[(b)] $X$ has the Schur property.

\end{enumerate}

\end{proposition}

\begin{proof}

 (b) $\Rightarrow$ (a) is straightforward.

$\bf{(a) \Rightarrow (b):}$  Suppose that $X$ does not have the Schur property then there exists $x_n \rightharpoonup 0, \| x_n \| \leq 1 \mbox{ for all } n \mbox{ with } \lim_{n \rightarrow \infty} \| x_n  \| = \alpha > 0.$

Consider $x \in X, \| x \| = a, RW(a, X) = a.$  Then
\[\liminf_{n \rightarrow \infty}\| x_n + x \| \wedge \liminf_{n \rightarrow \infty}\| x_n - x \| \leq \| x \|.\]
 
\begin{enumerate}

\item $\liminf_{n \rightarrow \infty}\| x_n + x \| \leq \liminf_{n \rightarrow \infty}\| x_n - x \|$

\bigskip

So $\liminf_{n \rightarrow \infty}\| x_n + x \| \leq \| x \|.$

\bigskip

Weak lower semicontinuity of the norm means 
\[\| x \| \leq \liminf_{n \rightarrow \infty}\| x_n + x \|.\]

Thus $\liminf_{n \rightarrow \infty}\| x_n + x \|  = \| x \|$ but $x_n + x \rightharpoonup x.$

\bigskip

The Kadec-Klee norm implies $x_n + x \rightarrow x$ which means $x_n \rightarrow 0.$  

\bigskip

This contradicts $\lim_{n \rightarrow \infty} \| x_n  \| = \alpha > 0.$

\bigskip

\item $\liminf_{n \rightarrow \infty}\| x_n - x \| \leq \liminf_{n \rightarrow \infty}\| x_n + x \|$

\bigskip

The proof follows the same form as in (1) and results in a contradiction.

\end{enumerate}

Therefore $X$ has the Schur property.

\end{proof}

{\bf Remark} At this stage $R(a, X) = a$ cannot replace $RW(a, X) = a.$

{\bf Other relationships}

$RW(1,X) \leq J(X)$ where $J(X)$ is the James constant.

\begin{proof}

\[J(X) := \sup \{ \{ \| x + y \| \wedge \| x - y \| \}: \| x \|, \| y \| \leq 1 \}.\]

Let $\eta  > 0$ then there exists $x \in X, \| x \| \leq 1, x_n \rightharpoonup 0, \| x_n \| \leq 1 \mbox{ for all } n$ such that 

\[\liminf_{n \rightarrow \infty}\| x_n + x \| \wedge \liminf_{n \rightarrow \infty}\| x_n - x \| > RW(1, X) - \eta.\]

By extracting a subsequence, still denoted by $(x_n)$, we may assume 

\[\lim_{n \rightarrow \infty}\| x_n + x \| \wedge \lim_{n \rightarrow \infty}\| x_n - x \| > RW(1, X) - \eta.\]

But $J(X) \geq \| x_n + x \| \wedge \| x_n - x \| \mbox{ for all } n.$

Therefore $J(X) \geq \lim_{n \rightarrow \infty}\| x_n + x \| \wedge \lim_{n \rightarrow \infty}\| x_n - x \|.$

$J(X) > RW(1, X) - \eta.$

Allowing $\eta \rightarrow 0, J(X) \geq RW(1, X).$

\end{proof}

$[RW(1,X)]^2 \leq 2C_{NJ}(X)$ where $C_{NJ}(X)$ is the von-Neumann-Jordan constant.

\begin{proof}

\begin{align*}
C_{NJ} & := \sup \left \{ \frac{ \| x + y \|^2 + \| x - y \|^2 }{ 2\left ( \| x \|^2 + \| y \|^2 \right ) } : \| x \|  + \| y \| \not = 0   \right \} \\
& =  \sup \left \{ \frac{ \| x + y \|^2 + \| x - y \|^2 }{ 2\left ( \| x \|^2 + \| y \|^2 \right ) } : \| x \|  = 1, \| y \| \leq 1  \right \}.
\end{align*}

Following the same path as in the previous proof let $\eta > 0$ then there exists 
\[ x \in X, \| x \| \leq 1, x_n \rightharpoonup 0, \| x_n \| \leq 1 \mbox{ for all } n \]
such that 
\[ \liminf_{n \rightarrow \infty}\| x_n + x \| \wedge \liminf_{n \rightarrow \infty}\| x_n - x \| > RW(1, X) - \eta. \]

By extracting a subsequence, still denoted by $(x_n)$, we may assume 
\[ \lim_{n \rightarrow \infty}\| x_n + x \| \wedge \lim_{n \rightarrow \infty}\| x_n - x \| > RW(1, X) - \eta. \]

Now
\[ C_{NJ} \geq \frac{ \| x_n + x \|^2 + \| x_ n - x \|^2 }{ 2\left ( \| x_n \|^2 + \| x \|^2 \right ) } \mbox{ for all } n. \]

\begin{align*}
\mbox{ Therefore } C_{NJ} & \geq \frac{ \| x_n + x \|^2 + \| x_ n - x \|^2 }{ 2\left ( 1+ 1 \right ) } \mbox{ for all } n. \\
4C_{NJ} & \geq  \| x_n + x \|^2 + \| x_ n - x \|^2 \mbox{ for all } n \\
& \geq \lim_{n \rightarrow \infty}\| x_n + x \|^2 + \lim_{n \rightarrow \infty}\| x_n - x \|^2 \\
& \geq  2(\lim_{n \rightarrow \infty}\| x_n + x \| \wedge \lim_{n \rightarrow \infty}\| x_n - x \|)^2.
\end{align*}

So $2C_{NJ} \geq (RW(1, X) - \eta)^2.$

Allowing $\eta \rightarrow 0, 2C_{NJ} \geq RW(1, X) ^2.$

\end{proof}

Equality in both these inequalities is achieved in $l_p, 1 < p \leq 2.$

Consider $l_p, 1 < p \leq 2.$  Because $l_p$ has WORTH, $R(l_p) = RW(1, l_p) = 2^{1/p}.$

 Note that Banach space, $X$, has WORTH if for every weak null sequence and every $x \in X$
 \[ \lim_{n \rightarrow \infty} | \| x_n - x \| - \| x_n + x \| |  = 0. \]

This property was introduced by Sims [28], Rosenthal [26] and Cowell and Kalton [1]

$J(l_p) = 2^{1/p} = RW(1, l_p).$

$C_{NJ}(l_p) = 2^{2/p - 1} \mbox{ so } 2C_{NJ}(l_p) = 2 \times 2^{2/p - 1} = 2^{2/p} = RW(1, l_p)^2$.

See Saejung [27].

Similar inequalities appear involving $R(a, X), R(X), C_{NJ}(X), WCS(X) \mbox{ and }J(X)$ in Saejung [27] and Mazcu\~{n}\'{a}n-Navarro [21].

\begin{proposition}

Let $X$ be a separable Banach space then if $X$ is uniformly nonsquare then $X$ has the FPP.

\end{proposition}

\begin{proof}

Either use uniformly nonsquare is equivalent to $J(X) < 2$ or is equivalent to $C_{NJ}(X) < 2.$

\end{proof}
 
\section{$R(X) = 1, R(1, X) = 1, RW(1, X) = 1$}
 
 An important Banach space property is Property($M$) where  whenever $x_n \rightharpoonup 0 $
 then $\limsup_{\rightarrow \infty}\| x_n - x \|$ is a function of $\| x \|$ only.  An equivalent definition is

\begin{definition}

A Banach space $X$ has Propery($M$) if whenever $x_n \rightharpoonup 0 \mbox{ and } \newline \| u \| \leq \| v \| \mbox{ then } \limsup_{n \rightarrow \infty} \| x_n + u \| \leq \limsup_{n \rightarrow \infty} \| x_n + v \|.$ 

\end{definition}

It was shown in [13] that Property($M$) implies the w-FPP.  

Property($M^*$) in $X^*$ is the same except that it involves weak* null sequences.

In [2] it was shown that $R(X) = 1$ implies $X$ has property($M$).  A crucial step in the proof was to show that $R(X) = 1$ implies $X$ has the nonstrict Opial condition.   That is
\[ \mbox{ if } x_n \rightharpoonup 0, \mbox{ then } \limsup_{n \rightarrow \infty}\| x_n \| \leq\limsup_{n \rightarrow \infty}\| x_n + x \| \mbox{ for all } x. \]

If the sequence, $(x_n)$, is in the dual, $X^*$, and the convergence is weak* then the condition is called the nonstrict *Opial condition.

Dhompongsa and Kaewkhao in [6] improved the result from [2] by showing that $R(X) = 1$ if and only if $X$ has property($m_\infty$).  

Here property($m_\infty$) is if $x_n \rightharpoonup 0$ and $x \in X$ then 
\[ \limsup_{n \rightarrow \infty}\| x_n + x \| =  \limsup_{n \rightarrow \infty}\| x_n \| \vee \| x \|. \]

This property was introduced in [16] and formally named as such in [17].  In $X^*$, property($m_\infty^*$) is defined in the same way but with weak* null sequences.

\begin{proposition} 

Let $X$ be a separable Banach space with $R(X) = 1$ then $c_0 \hookrightarrow X.$

\end{proposition}

\begin{proof}

In [2] it was shown that if $X$ is a separable Banach space with Property($M$) then $c_0  \hookrightarrow X$ if and only if there exists $x_n \rightharpoonup 0, \lim_{n \rightarrow \infty}\| x_n \| = 1$ such that 
\[ \lim_{n \rightarrow \infty}\| x + t x_n \| = 1 \vee t  \]

for all $t \geq 0, \mbox{ and all } x, \mbox{ with } \| x \| = 1.$ 

Property($m_\infty$) ensures that this condition is satisfied for all weak null sequences.

\end{proof}

{\bf Remarks} 

\begin{enumerate}

\item[1.]  Properties that $X$ does NOT posses when $R(X) = 1$ are: weak normal structure, Kadec-Klee norm, the Schur property, property($K$), reflexivity and Opial's condition.  See [2] for the reasoning.

\item[2.] If $\ell_1 \hookrightarrow X$ nothing can be said.  Most of the results from Kalton [16] and Kalton and Werner [17] require $\ell_1 \not \hookrightarrow X.$  So if in addition to $R(X) = 1$ we have $\ell_1 \not \hookrightarrow X$ then $X^*$ has Property($M^*$), has weak * Kadec-Klee norm, is separable so $c_0 \not \hookrightarrow X^*$.  $X^*$ also has property($m_1^*$) which means for every weak* null sequence, $(x_n^*),$ and every $x^* \in X$ we have

\[ \limsup_{n \rightarrow \infty}\| x_n ^* +  x^* \| = \limsup_{n \rightarrow \infty}\| x_n^* \|  + \| x^* \|. \]

This property implies that $X^*$ has the uniform Opial condition and also $R(X^*) = 2.$

Note, for future reference, that propery($m_1$) in $X$ means that for every weak null sequence, $(x_n) \mbox{ and } x \in X$

\[ \limsup_{n \rightarrow \infty}\| x_n + x \| = \limsup_{n \rightarrow \infty}\| x_n \|  + \| x \|. \]

If $X^*$ has Property($M^*$) then $X$ is $M$-embedded and $X^*$ is an $L$-summand.  An $M$-embedded nonreflexive Banach space, $X,$ is not weak sequentially complete and fails the Radon-Nikodym property.  Plus every copy of $c_0$ in $X$ is complemented.  $X$ being $M$-embedded also means that $X^*$ is weak sequentially complete and contains a complemented copy of $\ell_1.$
 
\item[3.] Kalton and Werner [17] have shown that if $X$ is a separable Banach space not containing a copy of $\ell_1$ then $X$ has property($m_\infty$) if and only if for every $\epsilon > 0,$ there is a subspace $X_0$ of $c_0$ with $d(X, X_0) < 1 + \epsilon.$  Thus when $R(X) = 1$ and $\ell_1 \not \hookrightarrow X$ then $X$ contains a complemented copy of $c_0$ and is almost isometric to a subspace of $c_0.$

\end{enumerate}

 At this stage, assuming either $R(1, X) = 1$ or $RW(1, X) = 1$ instead of $R(X) = 1$, that crucial step of $X$ possessing the nonstrict Opial condition cannot be proven.  So it is an open problem: If either $R(1, X) = 1$ or $RW(1, X) = 1$ does $X$ have Property($M$)?
 
But in the mean time we can state the following propositions.

\begin{proposition} 

If $X$ be a separable Banach space with $RW(1, X) = 1$ and satisfies the nonstrict Opial condition then $X$ has Property($M$).  In particular, it has property($m_\infty$).

\end{proposition}

\begin{proof}

Assume $X$ has the nonstrict Opial condition and $RW(1, X) = 1.$  Consider $x_n \rightharpoonup 0 \mbox{ and } x \in X.$  Using the weak lower semi-continuity of the norm we have
\[  \| x \| \leq  \limsup_{n \rightarrow \infty}\| x_n  \pm  x \|.  \]

The nonstrict Opial condition leads to
\[  \limsup_{n \rightarrow \infty}\| x_n \| \leq  \limsup_{n \rightarrow \infty}\| x_n  \pm  x \|. \]

Combining these two inequalities gives
\[ \| x \| \vee \limsup_{n \rightarrow \infty}\| x_n \| \leq \limsup_{n \rightarrow \infty}\| x_n  \pm  x \|. \]

Let $a = \| x \| \vee  \limsup_{n \rightarrow \infty}\| x_n \|.$   If $a = 0,$ then $m_\infty$ equation is trivially satisfied.  So assume $a \not = 0$ then
\[ \left \| \frac{ x }{ a} \right \| \leq 1, \limsup_{n \rightarrow \infty} \left \| \frac{ x_n }{ a } \right \| \leq 1 \mbox{ and } \frac{ x_n }{ a } \rightharpoonup 0. \]

By taking subsequences if necessary, we may assume that $\left \| \frac{ x_n }{ a } \right \| \leq 1 \mbox{ for all } n.$

Using $RW(1, X) = 1$,
\[ \limsup_{n \rightarrow \infty} \left \| \frac{ x_n }{ a} + \frac{ x }{ a } \right \| \wedge \limsup_{n \rightarrow \infty} \left \| \frac{ x_n }{ a} - \frac{ x }{ a } \right \| \leq 1. \]

Then 
\[ \limsup_{n \rightarrow \infty} \| x_n + x \| \wedge \limsup_{n \rightarrow \infty} \| x_n - x \|  \leq \| x \| \vee  \limsup_{n \rightarrow \infty}\| x_n \|. \]

But  $\| x \| \vee  \limsup_{n \rightarrow \infty}\| x_n \| \leq \limsup_{n \rightarrow \infty}\| x_n \pm  x \|$ leading to

\[ \limsup_{n \rightarrow \infty} \| x_n + x \| \wedge \limsup_{n \rightarrow \infty} \| x_n - x \| = \| x \| \vee  \limsup_{n \rightarrow \infty}\| x_n \|. \]

Then either 
\[\limsup_{n \rightarrow \infty} \| x_n + x \| = \| x \| \vee \limsup_{n \rightarrow \infty}\| x_n \| \mbox{ or } \limsup_{n \rightarrow \infty} \| x_n - x \| = \| x \| \vee \limsup_{n \rightarrow \infty}\| x_n \|. \]

Since this is true for all $x \in X$ changing $x$ into $-x$ means 
\[ \limsup_{n \rightarrow \infty}\| x_n  \pm  x \| = \| x \| \vee \limsup_{n \rightarrow \infty}\| x_n \|. \]

Thus $X$ has property($m_\infty$).

\end{proof}

For $R(1, X) = 1$ the situation is not as clear.  For a given weak null sequence $(x_n)$, if the norm of any element, $x,$ is big enough then $X$ behaves as if it has property($m_\infty$)  or even property($m_1$) or has the Schur property.  Below is an indication of what happens.  If the norm of $x$ is too small then no conclusion can be reached.  It is very curious.

Let $X$ have the nonstrict Opial condition and $R(1, X) = 1.$  Also consider \newline $x_n \rightharpoonup 0 \mbox{ and } x \in X.$  The first part follows the first section of the previous proof:
\[ \| x \| \vee \limsup_{n \rightarrow \infty}\| x_n \| \leq \limsup_{n \rightarrow \infty}\| x_n  \pm  x \|. \]

Let $a = \| x \| \vee  \limsup_{n \rightarrow \infty}\| x_n \|.$  If $a = 0,$ the $m_\infty$ equation is trivially satisfied.  So assume $a \not = 0$ then 
\[ \left \| \frac{ x }{ a} \right \| \leq 1, \limsup_{n \rightarrow \infty} \left \| \frac{ x_n }{ a } \right \| \leq 1 \mbox{ and } \frac{ x_n }{ a } \rightharpoonup 0. \]

By taking subsequences if necessary, we may assume that $\left \| \frac{ x_n }{ a } \right \| \leq 1 \mbox{ for all } n.$

Again using the weak lower semi-continuity of the norm, 
\[ \limsup_{n \rightarrow \infty}\| x_n - x_m \| \geq \| x_m \| \mbox{ for all } m. \]

So
\[ \limsup_{n \rightarrow \infty} \limsup_{m \rightarrow \infty}\| x_n - x_m \| \geq \limsup_{m \rightarrow \infty} \| x_m \|.   \quad \dag \]

If $\| x \| \geq \limsup_{n \rightarrow \infty}\limsup_{m \rightarrow \infty}\| x_n - x_m \|$ then using \dag

\[ \limsup_{n \rightarrow \infty} \limsup_{m \rightarrow \infty}\| x_n - x_m \| \leq \| x \| \vee \limsup_{n \rightarrow \infty} \| x_n \| = a. \]

Thus
\[ \limsup_{n \rightarrow \infty} \limsup_{m \rightarrow \infty} \left \| \frac { x_n }{ a} - \frac { x_m }{ a} \right \| \leq 1. \]

Using $R(1, X) = 1,$ 
\[ \limsup_{n \rightarrow \infty} \left \| \frac { x_n }{ a} + \frac { x }{ a} \right \| \leq 1. \]

Leading to $a \leq \limsup_{n \rightarrow \infty}\| x_n + x \| \leq a$ so
\[  \limsup_{n \rightarrow \infty}\| x_n + x \| = \| x \| \vee  \limsup_{n \rightarrow \infty}\| x_n \|. \]

This means that for a given weak null sequence, $(x_n),$ if any element, $x,$ of $X$ is such that $\| x \| \geq \limsup_{n \rightarrow \infty}\limsup_{m \rightarrow \infty}\| x_n - x_m \|$ then the property($m_\infty$) equation is satisfied.  So an interesting situation.  Note that if $\lim_{n \rightarrow \infty} x_n = 0$ then the property($m_1$) equation is also satisfied.  This would be true if $X$ has the Schur property.  Interestingly, in [2] the author showed that the Schur property is equivalent to $X$ satisfying both property($m_\infty$) and property($m_1$).

Close to the origin there is no structure associated with weak null sequences.  Away from the origin a very different scenario appears; $X$ has at least property($m_\infty$), maybe the Schur property.  Are there any other Banach spaces where this sort of thing occurs - irregularity near the origin, strong structure away from the origin?

Another important Banach space property, in the context of fixed point theory, is Property($K$) which was introduced in [29].  

\begin{definition}

A Banach space $X$ has property($K$) if there exists $K \in [0, 1)$ such that whenever $x_n \rightharpoonup 0, \lim_{ n \rightarrow \infty} \| x_n \| = 1 \mbox{ and }  \liminf_{n \rightarrow \infty}\| x_n - x \|  \leq 1 \mbox{ then } \| x \| \leq K.$

\end{definition}

Note that implicit in this definition is the fact that $X$ cannot have the Schur propery.

Sims in [29] showed that Property($K$) implied weak normal structure and hence the w-FPP.

If $R(X) = 1$ then $X$ cannot have Property($K$) has shown below.

\begin{proposition}

Let $X$ be a separable Banach space then if $X$ has Property($K$) then $R(X) > 1$.

\end{proposition}
 
\begin{proof}

Let $X$ have Property($K$) and assume that $R(X) = 1.$  Then by proposition 5.2, $c_0  \hookrightarrow X.$  In [5] Dalby and Sims showed that if $c_0 \hookrightarrow X$ then $X$ does not have Property($K$).  We have the required contradiction.

\end{proof}

The same applies to $RW(1, X).$

\begin{proposition}

If a separable Banach space, $X$, has Property($K$) then \newline $RW(1, X) > 1$.

\end{proposition}

\begin{proof}

Let $X$ have Property($K$) and assume that $RW(X) = 1.$

Consider $x_n \rightharpoonup 0, \| x_n \| =1 \mbox{ for all } n$ and $x \in X$ where $\| x \| = 1$ then $RW(1, X) = 1$ implies
\[ \liminf_{n \rightarrow \infty}\| x_n - x \| \wedge \liminf_{n \rightarrow \infty}\| x_n + x \| \leq RW(1, X) = 1. \]

So $\liminf_{n \rightarrow \infty}\| x_n + x \|  \leq 1 \mbox{ or } \liminf_{n \rightarrow \infty}\| x_n - x \|  \leq 1.$  Either way, Property($K$) implies $1 = \| x \| \leq K < 1,$ a contradiction.

\end{proof}

A similar proof can be used for the case of $R(1, X)$.

\begin{proposition}

If a separable Banach space, $X$, has Property($K$) then \newline $R(1, X) > 1$.

\end{proposition}

\begin{proof}

Let $X$ have Property($K$) and assume that $R(X) = 1.$

Consider $x_n \rightharpoonup 0, \| x_n \| =1 \mbox{ for all } n$ and $x \in X$ where $\| x \| = 1.$  Also let 
\[ \limsup_{n \rightarrow \infty} \limsup_{m \rightarrow \infty}\| x_n - x_m \| \leq 1.\]

Then $R(1, X) = 1$ implies  $\liminf_{n \rightarrow \infty}\| x_n + x \| \leq R(1, X) = 1.$

Property($K$) implies $1 = \| x \| \leq K < 1,$ a contradiction.

\end{proof}

{\bf Remark} Note that proposition 5.7 implies proposition 5.6 which in turn implies proposition 5.5.

\section{Results concerning the dual, $X^*$}

Next is bring the dual, $X^*$, into play.  In [4] it was shown that if $X^*$ has the nonstrict *Opial condition and $R(X^*) < 2$ then $X$ has Property($K$). This leads to the following.

\begin{corollary}

Let $X$ be a separable Banach space where $X^*$ has the nonstrict *Opial condition.  If $R(X) = 1$ then $R(X^*) = 2.$

\end{corollary}

At this stage $R(X) = 1$ cannot be replaced by either $RW(1, X) = 1$ or $R(1, X) = 1.$ in both the result mentioned before corollary 6.1 and  in corollary 6.1.  For the condition $RW(1, X) = 1,$ proposition 5.3 can be invoked to get:

\begin{proposition}

Let $X$ be a separable Banach space with the nonstrict Opial condition and $RW(1, X) = 1$.  If $\ell_1 \not \hookrightarrow X$ then $RW(1, X^*) = 2.$

\end{proposition}

\begin{proof}

If $X$ be a separable Banach space with the nonstrict Opial condition and $RW(1, X) = 1$ then by proposition 5.3, $X$ had property($m_\infty$).  Then $\ell_1 \not \hookrightarrow X$ implies $X^*$ has Property($M^*$) which implies WORTH* so $RW(1, X^*) = R(X^*).$  and the previous corollary can be used.  In fact $X^*$ has property($m_1^*$) and this property can be used to show that $RW(1, X^*) = 2.$

\end{proof}

The situation where $R(1, X) = 1$ is again unclear and the problem is an open one.

Earlier it was mentioned that $R(X) = 1$ is a strong enough condition to imply Property($M$).  It is straightforward to adapt that proof to show $R(X^*) = 1$ implies Property($M^*$). 

The proof used by Dhompongsa and Kaewkhao in [6] can again be easily transferred to the dual and weak* null sequences.  So $X^*$ has property($m_\infty^*$) if and only if $R(X^*) = 1.$

Next we investigate further the consequences of $R(X^*) = 1.$  First is the reverse of corollary 6.1.

\begin{proposition}

Let $X$ be a separable Banach space.  If $R(X^*) = 1$ then \newline $R(X) = 2.$

\end{proposition}

\begin{proof}

Consider $x_n \rightharpoonup 0, \lim_{n \rightarrow \infty}\| x_n \| = 1, \| x_n \| \leq 1 \mbox{ for all } n \mbox{ and } \| x \| = 1.$  Then $\liminf_{n \rightarrow \infty}\| x_n + x \|  \leq R(X).$

Let $x_n^* \in S_{X^*}, x_n^*(x_n) = \| x_n \| \mbox{ for all } n.$  Without loss of generality $x_n^* \overset {*}{\rightharpoonup} x^* \mbox{ where } \newline \| x^* \| \leq 1.$

Let $y^* \in S_{X^*}, y^*(x) = \| x \| = 1.$

Because $R(X^*) = 1, X^*$ has property($m_\infty^*$) and so
\[ \liminf _{n \rightarrow \infty}\| x_n^* - x^* + y^* \|  =  \liminf _{n \rightarrow \infty}\| x_n^* - x^* \| \vee \| y^* \| = \| y^* \| = 1. \]
Now
\begin{align*}
2 & \geq \liminf_{n \rightarrow \infty}\| x_n + x \|  \\
& \geq \liminf_{n \rightarrow \infty}(x_n^* - x^* + y^*)(x_n + x)  \\
& \geq  \liminf_{n \rightarrow \infty} (x_n^*)(x_n) + x^*(x) - 0 - x^*(x) + 0 +y^*(x) \\
& = 2.
\end{align*}

Therefore $\liminf_{n \rightarrow \infty}\| x_n + x \| = 2.$  But $R(X) \geq \liminf_{n \rightarrow \infty}\| x_n + x \|.$  Therefore $R(X) = 2.$

\end{proof}

\begin{corollary}

Let $X$ be a separable Banach space.  If $R(X) < 2$ then $R(X^*) > 1.$

\end{corollary}

{\bf Remark}  Because it is not known if $R(1, X^*) = 1 \mbox{ or } RW(1, X^*) = 1$ leads to $X^*$ having the nonstrict *Opial condition, $R(X^*) = 1$ in the proposition cannot be replaced with either of these two other properties.  In the case of $RW(1, X^*) = 1$ adding nonstrict *Opial leads to a similar result courtesy of proposition 5.3.

\begin{proposition}

Let $X$ be a separable Banach space.  If $RW(1, X^*) = 1$ and $X^*$ have the nonstrict *Opial condition then $RW(1, X) = 2.$

\end{proposition}

\begin{proof}

Proposition 5.3 transferred to $X^*$ means that $X^*$ has property($m_\infty$) and the proof follows the same lines as the previous one.
 
\end{proof}

\begin{corollary}

Let $X$ be a separable Banach space and $X^*$ have the nonstrict *Opial condition.  If $RW(1, X) < 2$ then $RW(1, X^*) > 1.$

\end{corollary}

The situation with $X^*$ having the nonstrict *Opial condition and $R(1, X) = 1$ is unresolved.  Please see the discussion after the proof of proposition 5.3.

Some of of the other consequences of $R(X^*) = 1$ are listed in the next proposition.

\begin{proposition}

Let $X$ be a separable Banach space with $R(X^*) = 1$ then

\begin{enumerate}

\item[(a)] $X^*$ has Property($M^*$).

\item[(b)] $X^*$ has a WORTH*.

\item[(c)] $X^*$ has the w*FPP.

\item[(d)] $c_0 \hookrightarrow X^*$.

\item[(e)] $X^*$ has the Radon-Nikodym Property ($RNP$).

\item[(f)] $X$ has Property($M$).

\item[(g)]  $X$ is an $M$-ideal in $X^{**}.$  ($X$ is $M$-embedded.)

\item[(h)] $X^*$ is $L$-summand in $X^{***}.$  ($X^*$ is $L$-embedded.)

\item[(i)] $\ell_1  \hookrightarrow X^*$

\item[(j)] $\ell_1  \hookrightarrow X$

\item[(k)] $c_0  \hookrightarrow X.$

\item[(l)] $X$ has the wFPP.

\item[(m)] $X$ is an Asplund space.

\end{enumerate}

\end{proposition}

\begin{proof}\quad

\begin{enumerate}

\item[(a)] As mentioned this was proved for $R(X) = 1$ in [2].  It also follows from property($m_\infty^*$).

\item[(b)] This is an easy consequence of Property($M^*$).

\item[(c)] This follows from Property($M^*$) and the result from [13].

 \item[(d)]  The proof in proposition 5.2 transfers easily to $X^*$.  Property($m_\infty^*$)  ensures that this condition is satisfied.

\item[(e)] In [19] Lima showed that WORTH* implied the $RNP.$

\item[(f)] That Property($M^*$) implies Property($M$) was proved in [16].

\item[(g)]  Again, this was proved in [16].

\item[(h)] It is well known that if $X$ is $M$-embedded then $X^*$ is $L$-embedded.  See [15] for example.

\item[(i)] Since $c_0  \hookrightarrow X^*, X$ is nonreflexive.  A nonreflexive Banach space that is $M$-embedded has a complemented copy of $\ell_1$ in $X^*.$  See [15] for example.

\item[(j)] Kalton and Werner in [17] showed that if $X$ was separable, has Property($M$) and $\ell_1 \not \hookrightarrow X$ then $X^*$ was separable.  But $c_0$ cannot be a subspace of a separable dual space.  See [20] for example.  This contradicts point (4).  So  $\ell_1  \hookrightarrow X.$

\item[(k)] Again from [15], any nonreflexive Banach space that is $M$-embedded contains a complemented copy of $c_0.$

\item[(l)]  This was proved in [13] using Property($M$).

\item[(m)] $X^*$ having the $RNP$ is equivalent to $X$ being an Asplund space.

\end{enumerate}

\end{proof}

{\bf Remarks} 

\begin{enumerate}

\item[1.] Those separable Banach spaces where $R(X^*) = 1$ form a collection of Banach spaces that have the wFPP when $R(X) = 2.$ 

\item[2.] It is of interest to mention some of the other properties that $X^*$ or $X$ do NOT have.  $X^*$ does not have weak* normal structure nor Property($K^*$) nor a weak* Kadec-Klee norm.  Also $X^*$ is not separable, is not reflexive and does not have the weak* Schur property.  $X$ does not have weak normal structure nor Kadec-Klee norm, is not reflexive and does not have the Schur property.  See [2] for details.

\item[3.] Both $X$ and $X^*$ fail the FPP.  This because Pfitzner showed in [23] that every nonreflexive subspace of an $M$-embedded Banach space contains an asymptotic isometric copy of $c_0$ and every nonreflexive subspace of an $L$-embedded Banach space contains an asymptotic isometric copy $\ell_1.$  From these properties flows the failure of the FPP.  See [8] and [9].

\item[4.] Thus separable Banach spaces, $X,$ where $R(X^*) = 1$ form a family of spaces that have the wFPP but not the FPP.  Similarly, the $X^*s$ form a family of spaces that have the w*FPP but not the FPP.

\item[5.] The results of this proposition and the preceding remarks will still remain valid if $R(X^*) = 1$ is replaced by $X^*$ has Property($M^*$) and $c_0 \hookrightarrow X^*$.

\item[6.] More on Property($M^*$) and reflexivity can be found in [3].

\item[7.]  If $R(X^*) = 1$ is changed to $RW(1, X^*) = 1$ and $X^*$ has the nonstrict *Opial condition then the same list of properties is valid.

\item[8.] The result of  changing $R(X^*) = 1$ to $R(1, X^*) = 1$ and $X^*$ has the nonstrict *Opial condition remains open.

\end{enumerate}


\begin{thebibliography}{99}

\bibitem{1} S. R. Cowell and N. J. Kalton, {\it Asymptotic unconditionality}, Quarterly J. Math. {\bf 61} (2010), 217-240.

\bibitem {2} T. Dalby, {\it Relationships between properties that imply the weak fixed point property}, J. Math. Anal. Appl. {\bf 253} (2001), 578-589.

\bibitem {3} T. Dalby, {\it Reflexivity and the fixed point property in Banach spaces}, Research and Technical Reports, School of Mathematics, Statistics and Computer Science, The University of New England, Australia, (2002).

\bibitem {4} T. Dalby, {\it The effect of the dual on a Banach space and the weak fixed point property}, Bull. Austral. Math. Soc. {\bf 67} (2003), 177-185.

\bibitem {5} T. Dalby and B. Sims, {\it Banach lattices and the weak fixed point property}, In Proceedings of the Seventh International Conference on Fixed Point theory and its Applications, Guanajuato, Mexico (2005), 63-71.

\bibitem {6} S. Dhompongsa and A. Kaewkhao, {\it A note on properties that imply the fixed point property},  Abstr. Appl. Anal. 2006 (2006) article 34959, 12 pp.

\bibitem {7} T. Dom\'{i}nguez-Benavides, {\it A geometric coefficient implying the fixed point property and stability results}, Houston J. Math. {\bf 22} (1996), 835-849.

\bibitem {8} P. N. Dowling and C. J. Lennard, {\it Every nonreflexive subspace of $L_1[0,1]$ fails the fixed point property}, Proc. Amer. Math. Soc. {\bf 125} (1997), 443-446.

\bibitem {9} P. N. Dowling, C. J. Lennard and B. Turett, {\sl Reflexivity and the fixed-point property for nonexpansive maps}, J. Math. Anal. Appl. {\bf 200} (1996), 653-662.

\bibitem {10} J. Garc\'{i}a-Falset, {\it Stability and fixed points for nonexpansive mappings}, Houston J. Math. {\bf 20} (1994), 495-505.

\bibitem {11} J. Garc\'{i}a-Falset, {\it The fixed point property in Banach spaces with the NUS-property}, J. Math. Anal. Appl. {\bf 215} (1997), 532-534.

\bibitem {12} J. Garc\'{i}a-Falset, E. Llorens-Fuster and E. M. Mazcu\~{n}an-Navarroa, {\it Uniformly nonsquare Banach spaces have the fixed point property for nonexpansive mappings}, J. Funct. Anal. {\bf 233} (2006), 494-534.

\bibitem {13} J. Garc\'{i}a-Falset and B. Sims, {\it Property(M) and the weak fixed point property}, Proc. Amer. Math. Soc. {\bf 125} (1997), 2891-1896.

\bibitem {14} K. Goebel and W. A. Kirk, {\it Topics in metric fixed point theory}, Cambridge University Press, Cambridge, 1990.

\bibitem {15} P. Harmand, D. Werner and W. Werner, {\it M-ideals in Banach spaces and Banach algebras}, Lecture Notes in Mathematics, {\bf 1547}, Springer-Verlag, Berlin, 1993.

\bibitem {16} N. Kalton, {\it M-ideals of compact operators}, Illinois J. Math. {\bf 208} (1993), 147-169.

\bibitem {17} N. Kalton and D. Werner, {\it Property(M), M-ideals, and almost isometric structure of Banach spaces}, J. Reine Angew. Math. {\bf 461} (1995), 137-178.

\bibitem {18} W. A. Kirk and B. Sims (ed.), {\it Handbook of metric fixed point theory}, Kluwer Academic Publishers, Dordrecht,  2001.

\bibitem {19} \r{A}. Lima, {\it Property(wM*) and the unconditional metric compact approximation property}, Studia Math. {\bf 113} (1995), 249-263.
 
 \bibitem {20} J. Lindenstrauss and T. Tzafriri, {\it Classical Banach spaces I, Sequence Spaces}, Ergebnisse der Mathematik und ihrer Grenzebiete, {\bf 92}, Springer-Verlag, Berlin, 1977.
 
\bibitem {21} E. M. Mazcu\~{n}an-Navarroa, {\it Geometry of Banach spaces in metric fixed point theory}, Ph. D thesis, Department of Mathematical Analysis, Universitat de Valencia, Valencia, 1990. 

\bibitem {22} E. M. Mazcu\~{n}an-Navarroa, {\it Banach space properties sufficient for normal structure}, J. Math. Anal. Appl. {\bf 337} (2008), 197-218.

\bibitem {23} H. Pfitzner, {\it A note on asymptotically isometric copies of $\ell^1$ and $c_0$}, Proc. Amer. Math. Soc. {\bf 129} (2001), 1367-1373.

\bibitem {24} S. Prus, {\it Banach spaces with the uniform Opial condition}, Nonlinear Anal. {\bf 18} (1992), 697-704.

\bibitem {25} S. Prus, {\it Properties of direct sums of Banach spaces},  12th International Conference on Fixed Point Theory and its Applications (2017) Newcastle, Australia.

\bibitem {26} H. Rosenthal, {\it Some remarks concerning unconditional basic sequences}, Longhorn Notes, Texas Functional Analysis Seminar 1982-1983 (1983), The University of Texas at Austin, 15-47.

\bibitem {27} S. Saejung, {\it On James and von Neumann-Jordan constants and sufficient conditions for the fixed point property}, J. Math. Anal. Appl. {\bf 323} (2006), 1018-1024.

\bibitem {28} B. Sims, {\it Orthogonality and fixed points of nonexpansive maps}, Proc. Centre Math. Anal. Austral. Nat. Univ. {\bf 20} (1988), 178-186.

\bibitem {29} B. Sims {\it A class of spaces with weak normal structure}, Bull. Austral. Math. Soc. {\bf 50} (1994), 523-528.



\end{thebibliography}
\end{document}